\documentclass[a4paper, 11pt, oneside,reqno]{article}

\usepackage[utf8]{inputenc}
\usepackage[english]{babel}
\usepackage{caption}
\usepackage{float}
\usepackage{graphicx}
\usepackage{array}
\usepackage{bm} 
\usepackage{hyperref} 

\usepackage[left=1in, right=1in, top=1in, bottom=1in, includefoot, headheight=13.6pt]{geometry}
\usepackage{url}
\usepackage{datetime}
\usepackage{times}
\usepackage{xcolor}
\usepackage{amsmath,amssymb}
\usepackage{booktabs} 
\usepackage[titletoc,title]{appendix} 
\usepackage{geometry} 

\usepackage{rotating} 
\usepackage{enumitem} 
\usepackage{changepage} 

\usepackage{cases} 

\usepackage[amsmath,thmmarks]{ntheorem}

\newtheorem{theorem}{Theorem}

\newtheorem{remark}[theorem]{Remark}
\newtheorem{corollary}[theorem]{Corollary}
\newtheorem{lemma}[theorem]{Lemma}
\newtheorem{definition}[theorem]{Definition}

\newtheorem{proposition}[theorem]{Proposition}

\theoremstyle{nonumberplain}
\theorembodyfont{\normalfont}
\theoremsymbol{\ensuremath{\square}}
\newtheorem{proof}{Proof}

\makeatletter
\newcommand{\leqnomode}{\tagsleft@true}
\newcommand{\reqnomode}{\tagsleft@false}
\makeatother

\newdateformat{dateenglish}{\monthname[\THEMONTH]~\THEDAY,~\THEYEAR}
\DeclareMathOperator{\sgn}{sgn}
\DeclareMathOperator{\Var}{Var} 
\DeclareMathOperator{\Cov}{Cov}

\DeclareMathOperator{\E}{\mathbb{E}}
\def\N{\mathbb{N}}
\def\R{\mathbb{R}}
\def\Z{\mathbb{Z}}
\newcommand{\1}{\mbox{1\hspace{-0.28em}I}}
\newcommand\numberthis{\addtocounter{equation}{1}\tag{\theequation}}

\title{Bivariate FCLT for the Sample Quantile and Measures of Dispersion for Augmented GARCH($p$, $q$) processes}

\author{\Large Marcel Br\"autigam$\dagger$ $\ddagger^{\;\ast}$ and Marie Kratz$\dagger^{\;\ast}$\\[1ex]
\small $\dagger$ ESSEC Business School Paris, CREAR \\ \small $\ddagger$ Sorbonne University, LPSM \\ \small $^{\;\ast}$ LabEx MME-DII}

\date{}

\begin{document}
\maketitle

\begin{abstract}
\noindent
In this paper, we build upon the asymptotic theory for GARCH processes, considering the general class of augmented GARCH($p$, $q$) processes. Our contribution is to complement the well-known univariate asymptotics by providing a joint (bivariate) functional central limit theorem of the sample quantile and the r-th absolute centred sample moment. 
This extends existing results in the case of identically and independently distributed random variables.

\noindent We show that the conditions for the convergence of the estimators in the univariate case suffice even for the joint bivariate asymptotics.
We illustrate the general results with various specific examples from the class of augmented GARCH($p$, $q$) processes and show explicitly under which conditions on the moments and parameters of the process the joint asymptotics hold.

\bigskip
\noindent {\emph 2010 AMS classification}: 60F05; 60F17; 60G10; 62H10; 62H20
\end{abstract}

\section{Introduction and Notation}
\label{sec:Intro}

Since the introduction of the ARCH and GARCH processes in the seminal papers by Engle, \cite{Engle82}, and Bollerslev, \cite{Bollerslev86}, respectively, various GARCH modifications and extensions have been proposed and their statistical properties analysed (see e.g. \cite{Bollerslev08} for an (G)ARCH glossary).
Conditions for the stationarity of such processes, as well as central limit theorems (CLT) or functional central limit theorems (FCLT) have been obtained in various ways by exploiting the different dependence concepts underlying these GARCH type processes (see the introduction in \cite{Lee14} for references on CLT's under different dependence conditions). 

The limit theorems extend also to different estimators apart from the underlying process itself, as for example: powers of the process (e.g. \cite{Hormann08} for augmented GARCH($1$,$1$); \cite{Berkes08}, \cite{Lee14} for augmented GARCH($p$, $q$)), sample autocovariance and sample variance (e.g. \cite{Mikosch00} for the GARCH($1$,$1$); \cite{Aue06} for augmented GARCH($1$,$1$)), or the sample quantile.

Still, joint asymptotics of such estimators have not been considered yet. 
It is what we are developing in this paper, providing bivariate functional central limit theorems for the sample quantile together with the r-th absolute centred sample moment. This includes the case of the sample variance and also the sample mean absolute deviation around the sample mean (MAD), two well-known and widely used measures of dispersion, extending the results obtained in \cite{Brautigam19_iid} for identically and independently distributed (iid) random variables.

Note that the theoretical questions arised from previous studies in financial risk management, one (see \cite{Brautigam19}) where the correlation between a log-ratio of sample quantiles with the sample MAD is measured using log-returns from different stock indices, the other (see \cite{Zumbach18} and \cite{Zumbach12}) considering the correlation of ‘the realized volatilities with the centred volatility increment’ for different underlying processes.
Thus, we think that those asymptotic results may be of great use for applications in
statistics or other application fields. For instance, coming back to financial risk management and risk measure estimation, we could extend results obtained for the Value-at-Risk, when estimated by the sample quantile, to Expected Shortfall using once again the FCLT.

To cover a broad range of GARCH processes, we focus on so called augmented GARCH($p$, $q$) processes, introduced by Duan in \cite{Duan97}. They contain many well-known GARCH processes as special cases. Previous works on univariate CLT's and stationarity conditions for this class of GARCH processes are, inter alia, \cite{Aue06},\cite{Berkes08},\cite{Hormann08} and \cite{Lee14}.

The structure of the paper is as follows. We present in Section~\ref{sec:asympt_results} the main results about the bivariate FCLT for the sample quantile and the r-th absolute centred sample moment for augmented GARCH($p$,~$q$) processes. 
Then, we present specific examples of well-known GARCH models in Section~\ref{sec:exp} and show how the general conditions in the main result translate for these specific cases.
The proofs are given in Section~\ref{sec:proofs}.

{\sf \bf \large Notation}

We introduce the same notation as in \cite{Brautigam19_iid}: Let $(X_1,\cdots,X_n)$ be a sample of size $n$. Assuming the random variables $X_i$'s have a common distribution, denote their parent random variable (rv) $X$ with parent cumulative distribution function (cdf) $F_X$, (and, given they exist,) probability density function (pdf) $f_X$, mean $\mu$, variance $\sigma^2$, as well as, for any integer $r\geq 1$ the r-th absolute centred moment, $\mu(X,r) := \E[\lvert X - \mu \rvert^r$ and quantile of order $p$ defined as $q_X(p):= \inf \{ x \in \R: F_X(x) \geq p \}$.
We denote the ordered sample by $X_{(1)}\leq ...\leq X_{(n)}$.

We consider the sample estimators of the two quantities of interest, i.e. first the sample quantile for any order $p \in [0,1]$ defined as $ q_n (p) = X_{( \lceil np \rceil )}$, where $\lceil x \rceil =   \min{ \{ m \in \mathbb{Z}  : m \geq x \} }$, $\lfloor x \rfloor =   \max{ \{ m \in \mathbb{Z}  : m \leq x \} }$ and $[x]$, are the rounded-up, rounded-off integer-parts and the nearest-integer of a real number $x \in \R$, respectively. Second, the r-th absolute centred sample moment defined, for $r \in \N$, by
\begin{equation}\label{eq:sampleMDisp}
\hat{m}(X,n,r) := \frac{1}{n} \sum_{i=1}^n \lvert X_i -  \bar{X}_n \rvert^r,
\end{equation}
$\bar{X}_n$ denoting the empirical mean.
Special cases of this latter estimator include the sample variance ($r=2$) 
and the sample mean absolute deviation around the sample mean ($r=1$).

Recall the standard notation $u^T$ for the transpose of a vector $u$ and,
for the signum function, $\displaystyle \sgn(x) := -\1_{(x<0)}+\1_{(x>0)}$.
By $\lvert \cdot \rvert$ we denote the euclidean norm and the usual $L_p$-norm is denoted by $\| \cdot \|_p := \E^{1/p}[ \lvert \cdot \rvert^p]$.
Moreover the notations $\overset{d}\rightarrow$, $\overset{a.s.}\rightarrow$, $\overset{P}\rightarrow$ and $\overset{D_d[0,1]}\rightarrow$ correspond to the convergence in distribution, almost surely, in probability and in distribution of a random vector in the d-dimensional Skorohod space $D_d[0,1]$. Further, for real-valued functions $f,g$, we write $f(x) = O(g(x))$ (as $x \rightarrow \infty)$ if and only if there exists a positive constant $M$ and a real number $x_0$ s.t. $\lvert f(x) \rvert \leq M g(x)$ for all $x \geq x_0$, and $f(x)=o(g(x))$ (as $x \rightarrow \infty$) if for all $\epsilon>0$ there exists a real number $x_0$ s.t. $\lvert f(x) \rvert \leq \epsilon g(x)$ for all $x \geq x_0$. Analogously, for a sequence of rv's $X_n$ and constants $a_n$, we denote by $X_n = o_P(a_n)$ the convergence in probability to 0 of $X_n/a_n$.
%
\section{The Bivariate FCLT}
\label{sec:asympt_results}

\vspace{-1ex}
Let us introduce the augmented GARCH($p$, $q$) process $X=(X_t)_{t \in \Z}$, due to Duan in~\cite{Duan97}, namely, for integers $p \geq 1$ and $q\geq 0$, $X_t$ satisfies
\begin{align}
X_t &= \sigma_t ~\epsilon_t \label{eq:augm_GARCH_pq_1}
\\ \text{with~} \quad \Lambda(\sigma_t^2) &= \sum_{i=1}^p g_i (\epsilon_{t-i}) + \sum_{j=1}^q c_j (\epsilon_{t-j}) \Lambda(\sigma_{t-j}^2), \label{eq:augm_GARCH_pq_2}
\end{align}
where $(\epsilon_t)$ is a series of iid rv's with mean $0$ and variance $1$, $\sigma_t^2 = \Var(X_t)$ and $\Lambda, g_i, c_j, i=1,...,p, j=1,...,q$, are real-valued measurable functions. 
Also, as in \cite{Lee14}, we restrict the choice of $\Lambda$ to the so-called group of either polynomial GARCH($p$, $q$) or exponential GARCH($p$, $q$) processes (see Figure~\ref{fig:GARCH_nested} in the Appendix):
\[ (Lee) \quad \quad \Lambda(x) = x^{\delta}, \text{~for some~} \delta >0, \quad \text{~or~} \quad \Lambda(x) = \log(x).\]
Clearly, for a strictly stationary solution to~\eqref{eq:augm_GARCH_pq_1} and~\eqref{eq:augm_GARCH_pq_2} to exist, the functions $\Lambda, g_i, c_j$ as well as the innovation process $(\epsilon_t)_{t \in \Z}$ have to fulfill some regularity conditions (see e.g. \cite{Lee14}, Lemma 1).
Alike, for the bivariate FCLT to hold, certain conditions need to be fulfilled; we list them in the following.
\\ First, conditions concerning the dependence structure of the process $X$. We use the concept of $L_p$-near-epoch dependence ($L_p$-NED), 
using a definition due to Andrews in \cite{Andrews88} 
but restricted to stationary processes. 
Let $(Z_n)_{n \in \Z}$, be a sequence of rv's and $\mathcal{F}_s^t =  \sigma(Z_s,...,Z_t)$, for $s \leq t$, the corresponding sigma-algebra. 
Let us recall the $L_p$-NED definition.
\begin{definition}[$L_p$-NED, \cite{Andrews88}]
For $p>0$, a stationary sequence $(X_n)_{n \in \Z}$ is called $L_p$-NED on $\left( Z_n\right)_{ n\in \Z }$ if for $k \geq 0$
\[ \| X_1 - \E[X_1 \vert \mathcal{F}_{n-k}^{n+k}] \|_p \leq \nu(k), \]
for non-negative constants $\nu(k)$ such that $\nu(k) \rightarrow 0$ as $k \rightarrow \infty$.

If $\nu(k)= O(k^{-\tau -\epsilon})$ for some $\epsilon >0$, we say that $X_n$ is $L_p$-NED of size $\left(-\tau\right)$. \\ If $\nu(k)=O(e^{-\delta k})$ for some $\delta>0$, we say that $X_n$ is geometrically $L_p$-NED.
\end{definition}

The second set of conditions concerns the distribution of the augmented GARCH($p$, $q$) process. 
We impose three different types of conditions as in the iid case (see \cite{Brautigam19_iid}):
First, the existence of a finite $2k$-th moment for any integer $k>0$ for the innovation process $(\epsilon_t)$.
Then, given that the process $X$ is stationary, the continuity or $l$-fold differentiability of its distribution function $F_X$ (at a given point or neighbourhood) for any integer $l> 0$, 
and the positivity of its density $f_X$ (at a given point or neighbourhood). Those conditions are named as:
\begin{align*}
&(M_k) &&\E[\lvert \epsilon_0 \rvert^{2k}] < \infty, 
\\ &(C_0) && F_X \text{~is continuous}, 
\\ \phantom{text to make the distance less} &(C_l^{~'}) &&F_X \text{~is~} l\text{-times differentiable,} \phantom{text to make the distance between \& and \& \& less} 
\\ &(P) &&f_X \text{~is positive.} 
\end{align*}
The third type of conditions is set on the functions $g_i, c_j, i=1,...,p, j=1,...,q$ of the augmented GARCH($p$, $q$) process of the $(Lee)$ family: Positivity of the functions used and boundedness in $L_r$-norm for either the polynomial GARCH, $(P_r)$, or exponential/logarithmic GARCH, $(L_r)$, respectively, for a given integer $r>0$,\\
\begin{align*}
&(A) && g_i \geq 0, c_j \geq 0, i=1,...,p,~j=1,...,q,
\\ &(P_r) && \sum_{i=1}^p \| g_i(\epsilon_0) \|_r < \infty, \quad \sum_{j=1}^q \| c_j(\epsilon_0) \|_r < 1, 
\\ &(L_r) &&  \E[ \exp(4r \sum_{i=1}^p \lvert g_i(\epsilon_0) \rvert^2)] < \infty, \quad \sum_{j=1}^q \lvert c_j(\epsilon_0) \rvert < 1. 
\end{align*}
Note that condition $(L_r)$ requires the $c_j$ to be bounded functions.

\begin{remark}
By construction, from \eqref{eq:augm_GARCH_pq_1} and~\eqref{eq:augm_GARCH_pq_2}, $\sigma_t$ and $\epsilon_t$ are independent (and $\sigma_t$ is a functional of $(\epsilon_{t-j})_{j=1}^{\infty}$). Thus, the conditions on the moments, distribution and density could be formulated in terms of $\epsilon_t$ only. At the same time this might impose some conditions on the functions $g_i, c_j, i=1,...,p, j=1,...,q$ (which might not be covered by $(A)$, $(P_r)$ or $(L_r)$). Thus, we keep the conditions on $X_t$ even if they might not be minimal.
\end{remark}

Now, let us state the main result. To ease its presentation we introduce a trivariate normal random vector (functionals of $X$), $(U, V, W)^T$, with mean zero and the following covariance matrix:
\begin{equation*}
    (D)~\left\{\begin{aligned}
\Var(U) &= \Var(X_0) +2 \sum_{i=1}^{\infty} \Cov(X_i, X_0)
\\ \Var(V) &= \Var(\lvert X_0 \rvert^r) +2 \sum_{i=1}^{\infty}  \Cov(\lvert X_i\rvert^r, \lvert X_0 \rvert^r) 
\\ \Var(W) &= \Var \left( \frac{p- \1_{(X_0 \leq q_X(p))}}{f_X(q_X(p))} \right) + 2 \sum_{i=1}^{\infty} \Cov \left( \frac{p- \1_{(X_i \leq q_X(p))}}{f_X(q_X(p))}, \frac{p- \1_{(X_0 \leq q_X(p))}}{f_X(q_X(p))} \right)
\\ &= \frac{p(1-p)}{f_X^2(q_X(p))} + \frac{2}{f_X^2(q_X(p))} \sum_{i=1}^{\infty} \left( \E[\1_{(X_0 \leq q_X(p))} \1_{(X_i \leq q_X(p))} ] -p^2  \right) \notag
\\ \Cov(U,V) &= \sum_{i \in \mathbb{Z}} \Cov(\lvert X_i \rvert^r,X_0) = \sum_{i \in \Z} \Cov(\lvert X_0  \rvert^r, X_i  )
\\ \Cov(U,W) &= \frac{-1}{f_X(q_X(p))} \sum_{i \in \mathbb{Z}} \Cov( \1_{(X_i \leq q_X(p))},X_0) = \frac{-1}{f_X(q_X(p))} \sum_{i \in \Z} \Cov( \1_{(X_0 \leq q_X(p))},X_i )
\\ \Cov(V,W) &= \frac{-1}{f_X(q_X(p))} \sum_{i \in \mathbb{Z}} \Cov(  \lvert X_0 \rvert^r, \1_{(X_i \leq q_X(p))}) =\frac{-1}{f_X(q_X(p))}  \sum_{i \in \Z} \Cov(\lvert X_i \rvert^r, \1_{(X_0 \leq q_X(p))} ).
 \end{aligned}
    \right.
\end{equation*}
\begin{theorem}[bivariate FCLT] \label{thm:augm_GARCH_pq_asympt_qn}
For an integer $r>0$, consider an augmented GARCH($p$, $q$) process $X$ as defined in~\eqref{eq:augm_GARCH_pq_1} and~\eqref{eq:augm_GARCH_pq_2} satisfying condition $(Lee)$, $(C_0)$ at $0$ for $r=1$, and both conditions $(C_2^{~'}), (P)$ at $q_X(p)$.
Assume also conditions $(M_r),
 (A)$, and either $(P_{max(1,r/\delta)})$ for $X$ belonging to the group of polynomial GARCH, or $(L_r)$ for the group of exponential GARCH.
Introducing the random vector 
$T_{n,r}(X) = \begin{pmatrix} q_n(p) -q_X(p) \\ \hat{m}(X,n,r) -m(X,r) \end{pmatrix}$, we have the following FCLT: For $t \in [0,1]$, as $n\to\infty$,
\[  \sqrt{n}~t ~ T_{[nt],r}(X) \overset{D_2[0,1]}{\rightarrow}  \textbf{W}_{\Gamma^{(r)}} (t), \]
where $(\textbf{W}_{\Gamma^{(r)}}(t))_{t \in [0,1]}$ is the 2-dimensional Brownian motion with covariance matrix $\Gamma^{(r)} \in \R^{2\times 2}$ defined for any $(s,t) \in [0,1]^2$ by $\Cov(\textbf{W}_{\Gamma^{(r)}}(t),\textbf{W}_{\Gamma^{(r)}}(s)) = \min(s,t) \Gamma^{(r)}$, where
\begin{align*}
\Gamma_{11}^{(r)}&= \Var(W),
\\ \Gamma_{22}^{(r)} &= r^2 \E[ X_0^{r-1} \sgn(X_0)^r]^2 \Var(U) + \Var(V) - 2r \E[ X_0^{r-1} \sgn(X_0)^r] \Cov(U,V),
\\  \Gamma_{12}^{(r)}&= \Gamma_{21}^{(r)} = -r\E[ X_0^{r-1} \sgn(X_0)^r] \Cov(U,W) + \Cov(V,W),
\end{align*}
$(U, V, W)^T$ being the trivariate normal vector (functionals of $X$) with mean zero and covariance given in $(D)$, all series being absolute convergent.
\end{theorem}

\begin{remark}
Note that the bivariate FCLT between the sample quantile and the r-th absolute centred sample moment requires exactly the same conditions in comparison to the respective univariate convergence (which might be apparent after having gone through the proof of Theorem~\ref{thm:augm_GARCH_pq_asympt_qn}):

Requiring $(C_2^{~'}), (P)$ at $q_X(p)$ exactly correspond to the conditions for the CLT of the sample quantile of a stationary process which is $L_1$-NED with polynomial rate. 
Further, $(P_{\max{(1, r/\delta)}})$ or $(L_r)$ respectively, together with $(M_r), (A)$ and $(C_0)$ at $\mu$ for $r=1$, are the conditions for the univariate CLT of the r-th centred sample moment for augmented GARCH($p$, $q$) processes.

But this is no surprise, as the multivariate FCLT we apply is exactly based on proving univariate asymptotics and then deducing the multivariate convergence via Cram{\'e}r-Wold and a univariate tightness argument, see the proof of \cite{Aue09}.
\end{remark}
\vspace*{-0.3cm}
Choosing $t=1$ in Theorem~\ref{thm:augm_GARCH_pq_asympt_qn} provides the usual CLT that we state for completeness:
\begin{corollary}
Consider an augmented GARCH($p$, $q$) process as defined in~\eqref{eq:augm_GARCH_pq_1} and~\eqref{eq:augm_GARCH_pq_2}.
Under the same conditions as in Theorem~\ref{thm:augm_GARCH_pq_asympt_qn},
the joint behaviour of the sample quantile $q_n(p)$ (for $p \in (0,1)$) and the $r$-th absolute centred sample moment $\hat{m}(X,n,r)$, is asymptotically bivariate normal:
\begin{equation} 
\sqrt{n} \, \begin{pmatrix} q_n (p) - q_X(p) \\ \hat{m}(X,n,r)  - m(X,r) \end{pmatrix} \; \underset{n\to\infty}{\overset{d}{\longrightarrow}} \; \mathcal{N}(0, \Gamma^{(r)}), 
\end{equation}
where the asymptotic covariance matrix $\displaystyle \Gamma^{(r)}=(\Gamma_{ij}^{(r)}, 1\le i,j\le 2)$ is as in Theorem~\ref{thm:augm_GARCH_pq_asympt_qn}.
\end{corollary}
As special case we can also recover the CLT between the sample quantile and the r-th absolute centred sample moment in the iid case, given by Theorem 7 in~\cite{Brautigam19_iid}:
\begin{corollary}
Consider an augmented GARCH($p$, $q$) process as defined in~\eqref{eq:augm_GARCH_pq_1} and~\eqref{eq:augm_GARCH_pq_2}, choosing $g_i,c_j, \Lambda$ such that $\sigma_t^2 = \sigma^2 >0$ is a positive constant for all $t$. 
Under the same conditions as in Theorem~\ref{thm:augm_GARCH_pq_asympt_qn},
the joint behaviour of the sample quantile $q_n(p)$ (for $p \in (0,1)$) and the $r$-th absolute centred sample moment $\hat{m}(X,n,r)$, is asymptotically bivariate normal:
\begin{equation} 
\sqrt{n} \, \begin{pmatrix} q_n (p) - q_X(p) \\ \hat{m}(X,n,r)  - m(X,r) \end{pmatrix} \; \underset{n\to\infty}{\overset{d}{\longrightarrow}} \; \mathcal{N}(0, \Gamma^{(r)}), 
\end{equation}
where the asymptotic covariance matrix $\displaystyle \Gamma^{(r)}=(\Gamma_{ij}^{(r)}, 1\le i,j\le 2)$ simplifies to
\begin{align*}
\Gamma_{11}^{(r)}&= \frac{p(1-p)}{f_X^2(q_X(p))}; \quad \Gamma_{22}^{(r)} = r^2 \E[ X_0^{r-1} \sgn(X_0 )^r]^2 \sigma^2 + \Var(\lvert X_0  \rvert^r) - 2r \E[ X_0^{r-1} \sgn(X_0)^r] \Cov(X_0,\lvert X_0  \rvert);
\\  \Gamma_{12}^{(r)}&= \Gamma_{21}^{(r)} = \frac{1}{f_X(q_X(p))} \left( r \E[ X_0^{r-1} \sgn(X_0)^r] \Cov(\1_{(X_0 \leq q_X(p))},X_0) - \Cov(\1_{(X_0 \leq q_X(p))},\lvert X_0 \rvert^r) \right).
\end{align*}
\end{corollary}


\paragraph{Idea of the proof - }
Let us briefly describe the idea of the proof of Theorem~\ref{thm:augm_GARCH_pq_asympt_qn}, developed in Section~\ref{sec:proofs}. To prove the FCLT, we need to show that two conditions are fulfilled for the vector $T_{n,r}(X)$. These specific conditions arise from the application of the  multivariate FCLT (Theorem A.1 in \cite{Aue09}, which extends the univariate counterpart from, e.g., Billingsley in~\cite{Billingsley68}) and are the following: 
\begin{itemize}
\item[$(H_1)$] 
A representation of $T_{n,r}(X)$ given by the random vector $t_j$ with $\E[t_j] =0$ and $\| t_j^2 \|_2^2 < \infty$, $j=1,...,n$, such that we have, for all natural numbers $n$,
\begin{equation*} 
\frac{1}{n} \sum_{j=1}^n t_j = T_{n,r}(X) \text{~and~} t_j = f( \epsilon_j, \epsilon_{j-1},...),
\end{equation*}
where $f: \mathbb{R}^{\infty} \rightarrow \mathbb{R}^2$ is a measurable function and $({\epsilon}_j, j \in \mathbb{Z})$ is a sequence of real-valued iid rv's with mean $0$ and variance $1$.
\item[$(H_2)$]
A $\Delta$-dependent approximation of the random vector ${t}_j$ introduced in $(H_1)$, i.e. the existence of a sequence of random vectors $\left({t}_j^{(\Delta)}, j \in \Z\right)$ such that, for any $\Delta \geq 1$, we have
\begin{align*} 
&{t}_j^{(\Delta)} = {f}^{(\Delta)}( {\epsilon}_{j-\Delta},...,{\epsilon}_j, ..., {\epsilon}_{j+\Delta}), \text{~for measurable functions~} {f}^{(\Delta)}: \R^{2\Delta+1} \rightarrow \R^2,  
\\ &\text{and} \quad \sum_{\Delta \geq 1} \| {t}_0 - {t}_0^{(\Delta)} \|_2  < \infty.  
\end{align*}
\end{itemize} 
Checking the first condition, $(H_1)$, is done in two steps. First, we show why we can use the Bahadur representation of the sample quantile given in \cite{Wendler11} (Theorem 1). Second, we prove a corresponding representation for the r-th absolute centred sample moment (extending results from~\cite{Brautigam19_iid},~\cite{Segers14}). These representations of the sample quantile and r-th absolute centred sample moment then naturally fulfil $(H_1)$. 

To prove the second condition $(H_2)$, we need to find a $\Delta$-dependent approximation ${t}_j^{(\Delta)}$. For this, we show that the existence of our chosen ${t}_j^{(\Delta)}$ can be reduced to the existence of a $\Delta$-dependent approximation for the process $X_j$ or powers of the process $\lvert X_j \rvert$, for which results in \cite{Lee14} can be used.

\section{Examples} \label{sec:exp}
In this section we review some well-known examples of augmented GARCH($p$, $q$) processes and discuss which conditions these models need to fulfill in order for the bivariate asymptotics of Theorem~\ref{thm:augm_GARCH_pq_asympt_qn} to be valid.

Note that the moment condition on the innovations, $(M_r)$, as well as the continuity and differentiability conditions, $(C_2^{~'})$, $(P)$, each at $q_X(p)$, and $(C_0)$ at $0$ for $r=1$, remain the same for the whole class of augmented GARCH processes. But, depending on the specifications of the process, \eqref{eq:augm_GARCH_pq_1} and~\eqref{eq:augm_GARCH_pq_2}, the conditions, $(P_{\max(1,r/\delta)})$ for polynomial GARCH or $(L_r)$ for exponential GARCH respectively, translate differently in the various examples.

For this, we introduce in Table~\ref{tbl:vol_formula_models} different augmented GARCH($p$,$q$) models by providing for each the corresponding volatility equation, \eqref{eq:augm_GARCH_pq_2}, and the specifications of the functions $g_i$ and $c_j$. 
We consider 10 models which belong to the group of polynomial GARCH ($\Lambda(x) = x^{\delta}$) and two examples of exponential GARCH ($\Lambda(x) = \log(x)$). As the nesting of the different models presented is not obvious, we give a schematic overview in Figure~\ref{fig:GARCH_nested} in the Appendix.
An explanation of the abbreviations for, and authors of, the different models can be found there too. 
Note that the presented selection of augmented GARCH~($p$,$q$) processes is not exhaustive.

Note that in Table~\ref{tbl:vol_formula_models} the specification of $g_i$ is the same for the whole APGARCH family (only the $c_j$ change), whereas for the two exponential GARCH models, it is the reverse.
The general restrictions on the parameters are as follows: $\omega>0, \alpha_i\geq 0 , -1 \leq \gamma_i \leq 1, \beta_j \geq 0$ for $i=1,...,p$, $j=1,...,q$.
Further, the parameters in the GJR-GARCH (TGARCH) are denoted with an asterix (with a plus or minus) as they are not the same as in the other models (see the Appendix for details).
\begin{table}[H]
{\footnotesize
\caption{\label{tbl:vol_formula_models} \sf\small  Presentation of the volatility equation \eqref{eq:augm_GARCH_pq_2} and the corresponding specifications of functions $g_i,c_j$ for selected augmented GARCH models.}
\vspace{-3ex}
\begin{center}
\hspace*{-2.0cm}
\begin{tabular*}{570pt}{l | lll }
\hline
\\[-1.5ex]
 & standard formula for $\Lambda(\sigma_t^2)$ & corresponding specifications of $g_i,c_j$ in~\eqref{eq:augm_GARCH_pq_2} \\
\hline\hline
\\[-1.5ex]
\\[-2ex] \textbf{Polynomial GARCH}& & & 
\\[1.5ex]  ~~APGARCH family & $\displaystyle \sigma_t^{2 \delta} = \omega + \sum_{i=1}^p \alpha_i \left( \lvert y_{t-i} \rvert - \gamma_i y_{t-i} \right)^{2 \delta} + \sum_{j=1}^q \beta_j \sigma_{t-j}^{2 \delta}$ & \hspace*{-1.1cm}$g_i = \omega/p \;\text{and}\; c_j =  \alpha_j \left( \lvert \epsilon_{t-j} \rvert - \gamma_j \epsilon_{t-j} \right)^{2 \delta} + \beta_j $ 
\\[1.5ex] \quad\quad AGARCH & $\displaystyle\sigma_t^{2} = \omega + \sum_{i=1}^p \alpha_i \left( \lvert y_{t-i} \rvert - \gamma_i y_{t-i} \right)^{2 } + \sum_{j=1}^q \beta_j \sigma_{t-j}^{2}$ & \quad \quad ~$c_j =   \alpha_j \left( \lvert \epsilon_{t-j} \rvert - \gamma_j \epsilon_{t-j} \right)^{2} + \beta_j $
\\[1.5ex] \quad\quad GJR-GARCH & $\displaystyle\sigma_t^2 = \omega +  \sum_{i=1}^p \left( \alpha_i^{\ast} +\gamma_i^{\ast} \1_{(y_{t-i}<0)} \right)y_{t-i}^2 + \sum_{j=1}^q \beta_j \sigma_{t-j}^2 $ & \quad \quad ~$c_j = \beta_j + \alpha_j^{\ast} \epsilon_{t-j}^2 +\gamma_j^{\ast} \max(0,-\epsilon_{t-j})^2   $
\\[1.5ex] \quad\quad GARCH & $\displaystyle\sigma_t^2 = \omega + \sum_{i=1}^p \alpha_i y_{t-i}^2 + \sum_{j=1}^q \beta_j \sigma_{t-j}^2$ & \quad \quad ~$c_j = \alpha_j \epsilon_{t-j}^2 +\beta_j $
\\[1.5ex] \quad\quad ARCH & $\displaystyle\sigma_t^2 = \omega + \sum_{i=1}^p \alpha_i y_{t-i}^2$ & \quad \quad ~$c_j =  \alpha_j \epsilon_{t-j}^2 $
\\[1.5ex] \quad\quad TGARCH & $\displaystyle\sigma_t = \omega + \sum_{i=1}^p \left( \alpha_i^{+} max(y_{t-i},0)  - \alpha_i^{-} min(y_{t-i},0) \right) + \sum_{j=1}^q \beta_j \sigma_{t-j}$ & \quad \quad ~$ c_j = \alpha_j \lvert \epsilon_{t-j} \rvert  - \alpha_j \gamma_j \epsilon_{t-j} + \beta_j  $
\\[1.5ex] \quad\quad TSGARCH & $\displaystyle\sigma_t = \omega + \sum_{i=1}^p  \alpha_i \lvert y_{t-i} \rvert  + \sum_{j=1}^q \beta_j \sigma_{t-j}$ & \quad \quad ~$ c_j = \alpha_j \lvert \epsilon_{t-j} \rvert  + \beta_j $
\\[1.5ex] \quad\quad PGARCH & $\displaystyle\sigma_t^{\delta} = \omega + \sum_{i=1}^p \alpha_i \lvert y_{t-i} \rvert^{\delta}  + \sum_{j=1}^q \beta_j  \sigma_{t-j}^{\delta}.$ & \quad \quad ~$c_j =  \alpha_j \lvert \epsilon_{t-j} \rvert^{\delta}  + \beta_j $
\\[1.5ex] ~~VGARCH & $\displaystyle\sigma_t^2 = \omega + \sum_{i=1}^p \alpha_i (\epsilon_{t-i} + \gamma_i)^2 + \sum_{j=1}^q \beta_j \sigma_{t-j}^2.$ & $g_i = \omega/p + \alpha_i (\epsilon_{t-i} + \gamma_i)^2$ \; and \;$c_j = \beta_j$
\\[1.5ex] ~~NGARCH & $\displaystyle\sigma_t^2 = \omega + \sum_{i=1}^p \alpha_i (y_{t-i} + \gamma_i \sigma_{t-i})^2  + \sum_{j=1}^q \beta_j \sigma_{t-j}^2 $ & $g_i = \omega/p$\; and\; $c_j =  \alpha_j (\epsilon_{t-j} + \gamma_j)^2 +  \beta_j $
\\\hline
\\[-1.5ex] \textbf{Exponential GARCH} & & $c_j = \beta_j$\; and
\\[1.5ex] ~~MGARCH & $\displaystyle\log(\sigma_t^2) = \omega + \sum_{i=1}^p \alpha_i \log(\epsilon_{t-i}^2) + \sum_{j=1}^q \beta_j \log(\sigma_{t-j}^2)$ & ~~~~$g_i = \omega/p + \alpha_i \log(\epsilon_{t-i}^2)$
\\[1.5ex] ~~EGARCH & $\displaystyle\log(\sigma_t^2) = \omega + \sum_{i=1}^p \alpha_i \left(  \lvert \epsilon_{t-i} \rvert - \E\lvert \epsilon_{t-i}\rvert \right) + \gamma_i \epsilon_{t-i}  + \sum_{j=1}^q \beta_j \log(\sigma_{t-j}^2)$ & ~~~~$g_i = \omega/p + \alpha_i (\lvert\epsilon_{t-i}\rvert - \E\lvert \epsilon_{t-i} \rvert) + \gamma_i \epsilon_{t-i}$
\\  [1ex]
\hline
\end{tabular*}
\end{center}
}
\end{table}

In Tables~\ref{tbl:cond_bivariate-conv-11} and~\ref{tbl:cond_bivariate-conv-pq} we present how the conditions $(P_{\max(1,r/\delta)})$ or $(L_r)$ translate for each model.
Table~\ref{tbl:cond_bivariate-conv-11} treats the specific case of an augmented GARCH($p$, $q$) process with $p=q=1$ and is presented here whereas Table~\ref{tbl:cond_bivariate-conv-pq} treats the general case for arbitrary $p\geq1, q \geq0$ and is defered to the Appendix. 
In the first column we consider the conditions for the general r-th absolute centred sample moment, $r \in \N$.
Of biggest interest to us are the specific cases of the sample MAD ($r=1$) and the sample variance ($r=2$) as measure of dispersion estimators respectively, presented in the second and third column.

For the selected polynomial GARCH models the requirement ${\sum_{i=1}^p \| g_i(\epsilon_0) \|_{\max(1,r/\delta)} < \infty}$ in condition $(P_{\max(1,r/\delta)})$ will always be fulfilled. Thus, we only need to analyse the condition ${\sum_{j=1}^q \| c_j (\epsilon_0)  \|_{\max(1,r/\delta)} < 1}$.

Note that in Table~\ref{tbl:cond_bivariate-conv-11} (and also Table~\ref{tbl:cond_bivariate-conv-pq}) the restrictions on the parameter space, given by $(P_{\max(1,r/\delta)})$ or $(L_r)$ respectively, are the same as the conditions for univariate FCLT's of the process $X_t^r$ itself (see \cite{Berkes08}, \cite{Hormann08}). For $r=1$, they coincide with the conditions for e.g. $\beta$-mixing with exponential decay (see \cite{Carrasco02}).

\begin{table}[H]
{\small
\caption{\label{tbl:cond_bivariate-conv-11} \sf\small Conditions $(P_{\max{(1, r/\delta)}})$ or $(L_r)$ respectively translated for different augmented GARCH($1$,$1$) models. Left column for the general r-th absolute centred sample moment, middle for the MAD ($r=1$) and right for the variance ($r=2$).}
\vspace{-3ex}
\begin{center}
\hspace*{-2cm}
\begin{tabular}{l | lll }
\hline
\\[-1.5ex]
augmented \\ GARCH ($1$, $1$) & $r \in \N$ & $r=1$ & $r=2$ \\
\hline\hline
\\[-1.5ex]
\\[-2ex] APGARCH &  $ \E[\lvert \alpha_1 \left( \lvert \epsilon_0\rvert - \gamma_1 \epsilon_{t-1}\right)^{2 \delta} + \beta_1 \rvert^r] <1$ & $\alpha_1 \E\left[\left( \lvert \epsilon_0\rvert - \gamma_1 \epsilon_{t-1} \right)^{2\delta } \right]+ \beta_1  <1$ & $ \E[\lvert \alpha_1 \left(  \lvert \epsilon_0\rvert - \gamma_1 \epsilon_{t-1} \right)^{2 \delta} + \beta_1 \rvert^2] <1$
\\[1.5ex] AGARCH & $ \E[\lvert \alpha_1 \left( \lvert \epsilon_0\rvert - \gamma_1 \epsilon_{t-1}\right)^{2 } + \beta_1 \rvert^r] <1$ & $\alpha_1 \E\left[\left( \lvert \epsilon_0\rvert - \gamma_1 \epsilon_{t-1} \right)^{2 } \right]+ \beta_1  <1$ & $ \E[\lvert \alpha_1 \left( \lvert \epsilon_0\rvert - \gamma_1 \epsilon_{t-1}\right)^{2} + \beta_1 \rvert^2]<1$
\\[1.5ex] GJR-GARCH & $\E[\lvert \alpha_1^{\ast} \epsilon_0^2 + \beta_1 + \gamma_1^{\ast} \max(0,-\epsilon_0^2) \rvert^r] <1$ & $\alpha_1^{\ast} + \beta_1 + \gamma_1^{\ast} \E[\max(0,-\epsilon_0)^2] <1$ &  $\E[\lvert \alpha_1^{\ast} \epsilon_0^2 + \beta_1 + \gamma_1^{\ast} \max(0,-\epsilon_0^2) \rvert^2] <1$
\\[1.5ex] GARCH &  $\E[(\alpha_1 \epsilon_0^2 + \beta_1)^r]< 1 $ & $ \alpha_1 + \beta_1 < 1$ & $\alpha_1^2 \E[\epsilon_0^4] + \alpha_1 \beta_1 + \beta_1^2 < 1 $ 
\\[1.5ex] ARCH &  $\alpha_1^r \E[\epsilon_0^{2r}]< 1 $ & $ \alpha_1 < 1$ & $\alpha_1^2 \E[\epsilon_0^4] < 1 $ 
\\[1.5ex] TGARCH &  $\E[\lvert \alpha_1 \lvert \epsilon_{t-1} \rvert - \alpha_1 \gamma_1 \epsilon_{t-1} + \beta_1 \rvert^r] <1 $&  $\alpha_1 \E\lvert \epsilon_{t-1} \rvert + \beta_1 < 1$ &  $\E[\lvert \alpha_1 \lvert \epsilon_{t-1} \rvert - \alpha_1 \gamma_1 \epsilon_{t-1} + \beta_1 \rvert^2] <1$
\\[1.5ex] TSGARCH &  $\E[\lvert \alpha_1 \lvert \epsilon_{t-1} \rvert + \beta_1 \rvert^r] <1 $&  $\alpha_1 \E\lvert \epsilon_{t-1} \rvert + \beta_1 < 1$ &  $\E[\lvert \alpha_1 \lvert \epsilon_{t-1} \rvert + \beta_1 \rvert^2] <1$  
\\[1.5ex] PGARCH & $\E[\lvert \alpha_1 \lvert \epsilon_0 \rvert +\beta_1\rvert^{2r}] < 1$ & $\alpha_1  + 2 \alpha_1 \beta_1 \E\lvert \epsilon_0 \rvert +\beta_1^2 < 1$ & $ \E[\lvert \alpha_1 \lvert \epsilon_0 \rvert +\beta_1\rvert^{4}] < 1$
\\[1.5ex] VGARCH & \multicolumn{3}{c}{for any $r \in \N$: \quad $\beta_1 < 1$}
\\[1.5ex] NGARCH &  $\E[\lvert \alpha_1 (\epsilon_0 + \gamma_1)^2 + \beta_1 \rvert^r] < 1$ & $ \alpha_1 (1 + \gamma_1^2) + \beta_1 < 1$&  $ \E[\lvert \alpha_1 (\epsilon_0 + \gamma_1)^2 + \beta_1 \rvert^2] < 1$
\\[1.5ex] MGARCH & \multicolumn{3}{c}{for any $r \in \N$: \quad $\E[\exp(4r \lvert \omega/p + \alpha_1 \log(\epsilon_0^2)  \rvert^2)] < \infty$ and $\lvert \beta_1 \rvert <1$}
\\[1.5ex] EGARCH & \multicolumn{3}{c}{for any $r \in \N$: \quad $\E[\exp(4r  \lvert \omega/p + \alpha_1 (\lvert\epsilon_0\rvert - \E\lvert \epsilon_0 \rvert) + \gamma_1 \epsilon_0 \rvert^2)] < \infty$ and $\lvert \beta_1 \rvert <1$}
\\  [1ex]
\hline 
\end{tabular}
\end{center}
}
\end{table}

\section{Proofs} \label{sec:proofs}
Before stating the proof of the main theorem, let us start with two auxiliary results.
As it requires some work to find the asymptotics of $\hat{m}(X,n,r)= \frac{1}{n} \sum_{i=1}^n \lvert X_i - \bar{X}_n \rvert^r$ for any integer $r\geq 1$, and such a result is of interest in its own right, we give it separately in Proposition~\ref{prop:Abs_central_mom_asympt_garch}. To prove it, we need the following Lemma, which extends Lemma 2.1 in \cite{Segers14} (case $v=1$) to any moment $v \in \N$, and the iid case presented in Lemma~8 in \cite{Brautigam19_iid}.

\begin{lemma} \label{lemma:segers_modif_garch}
Consider a stationary and ergodic time-series $(X_n, n \geq 1)$ with parent rv $X$ which has `short-memory', i.e. $\sum_{i=0}^{\infty} \lvert \Cov(X_0, X_i) \rvert < \infty$. 
Then, for $v=1$ or $2$, given that the 2nd moment of $X$ exists, or, for any integer $v > 2$, given that the $v$-th moment of $X$ exists, letting $n \rightarrow \infty$, it holds that
\begin{equation} \label{eq:segers_lemma_eq_garch}
\frac{1}{n} \sum_{i=1}^n (X_i - \mu)^{v} \left( \lvert X_i - \bar{X}_n \rvert - \lvert X_i - \mu \rvert \right) = (\bar{X}_n - \mu) \times \E[(X - \mu )^v \sgn(\mu-X)] + o_P(1/\sqrt{n}).
\end{equation}
\end{lemma}

\begin{proof}
The proof follows the lines of its equivalent in the iid case; see proof of Lemma~8 in \cite{Brautigam19_iid}.
The argumentation needs to be adapted only at the end in two points, using the stationarity, ergodicity and short-memory of the process.
Here, it follows by these three properties
that $\sqrt{n} \lvert \bar{X}_n - \mu \rvert^{v+1} \underset{n \rightarrow \infty}{\overset{P} \rightarrow} 0$ holds for any integer $v \geq 1$. 
Further, as a last step, we use the ergodicity of the process, instead of the strong law of large numbers, to conclude that
\[ \frac{1}{n} \sum_{i=1}^n (X_i - \mu)^v \sgn(\mu-X_i) \underset{n \rightarrow \infty}{\overset{a.s.}\rightarrow} \E[(X - \mu)^v \sgn(\mu-X)]. \]
\end{proof}

Now we are ready to state the asymptotic relation between
the r-th absolute centred sample moment with known and unknown mean, respectively. 
This enables us to compute the asymptotics of $\hat{m}(X,n,r)$ (given that the necessary moments exist). As for Lemma~\ref{lemma:segers_modif_garch}, it is an extension to the stationary, ergodic and short-memory case of Proposition~9 in the iid case \cite{Brautigam19_iid}.
\begin{proposition} \label{prop:Abs_central_mom_asympt_garch}
Consider a stationary and ergodic time-series $(X_n, n \geq 1)$ with parent rv $X$ which has `short-memory', i.e. $\sum_{i=0}^{\infty} \lvert \Cov(X_0, X_i) \rvert < \infty$. 
Then, for any integer $r\geq 1$, given that the $r$-th moment of $X$ exists and $(C_0)$ at $\mu$ for $r=1$, it holds, as $n \rightarrow \infty$, that
\begin{equation} \label{eq:Abs_central_mom_asympt_garch}
\sqrt{n} \left(\frac{1}{n} \sum_{i=1}^n \lvert X_i - \bar{X}_n \rvert ^r \right) =  \sqrt{n} \left(\frac{1}{n} \sum_{i=1}^n \lvert X_i - \mu \rvert^r \right) - r \sqrt{n} (\bar{X}_n - \mu) \E[ (X - \mu)^{r-1} \sgn(X - \mu)^r] + o_P(1) .
\end{equation}
\end{proposition}

\begin{proof}
Analogously to the proof of Lemma~\ref{lemma:segers_modif_garch}, the proof can be extended from the proof of Proposition~9 in the iid case in \cite{Brautigam19_iid}.
We comment on the differences compared to the iid case for the three different cases of $r$: 

\textit{Even integers $r$ - }
Recall that for the corresponding result in the iid case (see the proof of Proposition~9 in \cite{Brautigam19_iid}) we refered to the example~5.2.7 in \cite{Lehmann99}. 
Therein they only consider the iid case but in this case, the argumentation still holds as $\sqrt{n} (\bar{X}_n -\mu)^v \overset{P}\rightarrow 0$, for $v\geq 2$, holds for an ergodic, stationary, short-memory process too.

\textit{Case $r=1$ - }
The result cited in the iid case 
holds for ergodic, stationary time-series too, see Lemma 2.1 in \cite{Segers14}.

\textit{Odd integer $r>1$ - }
We point out the three differences to the corresponding proof in the iid case.
First, as remarked above for even integers $r$,$\sqrt{n} (\bar{X}_n -\mu)^v \overset{P}\rightarrow 0$, for $v\geq 2$, follows from the stationarity, ergodicity and short-memory of the process.
Second, we use the ergodicity instead of the law of large numbers. Third, we use Lemma~\ref{lemma:segers_modif_garch} instead of its counterpart in the iid case, Lemma~8 in \cite{Brautigam19_iid}.
\end{proof}

Finally, we state a multivariate FCLT which we will use to prove the theorem, and which is from \cite{Aue09} (adapted to our needs):

\begin{lemma}[Theorem A.1 in \cite{Aue09}] \label{lemma:Aue_FCLT}
Consider a d-dimensional random process $(u_j, j\in \Z)$, 
which is centered and has finite variance, i.e.
\begin{equation} \label{eq:repr1}
\E[u_j] =0, \quad \| u_j \|_2^2 < \infty ~\forall j \in \Z,
\end{equation}
and has a causal (possibly non-linear) representation in terms of an iid process, i.e. 
\begin{equation} \label{eq:Aue09_representation_of_vector}
u_j = f( \epsilon_j, \epsilon_{j-1},...),
\end{equation}
where $f: \mathbb{R}^{\infty} \rightarrow \mathbb{R}^d$ is a measurable function and $({\epsilon}_j, j \in \mathbb{Z})$ is a sequence of real valued iid rv's with mean $0$ and variance $1$.

Suppose further, there exists a $\Delta$-dependent approximation of ${u}_j$, i.e. a sequence of d-dimensional random vectors $\left({u}_j^{(\Delta)}, j \in \Z\right)$ such that, for any $\Delta \geq 1$, we have
\begin{align} 
& u_j^{(\Delta)} ={f}^{(\Delta)}({\epsilon}_{j-\Delta},...,{\epsilon}_{j},...,{\epsilon}_{j+\Delta}) \label{eq:repr3}
\\ &\text{and}~\sum_{\Delta \geq 1} \| u_0 - u_0^{(\Delta)} \|_2 < \infty, \label{eq:repr2}
\end{align}
where ${f}^{(\Delta)}: \mathbb{R}^{2\Delta+1} \rightarrow \mathbb{R}^d$ is a measurable function. 

Then, the series $\Gamma = \sum_{j \in Z} \Cov( u_0, u_j)$ converges (coordinatewise) absolutely and an FCLT holds for $U_n := \frac{1}{n} \sum_{j=1}^n u_j$
\[ \sqrt{n}t U_{[nt]} \overset{D_d[0,1]}{\rightarrow} W_{\Gamma}(t),\]
where the convergence takes place in the d-dimensional Skorohod space $D_d[0,1]$ and  $(W_{\Gamma}(t), t \in [0,1])$ is a d-dimensional Brownian motion with covariance matrix $\Gamma$, i.e. it has mean 0 and 
\\ $\Cov(W_{\Gamma}(s), W_{\Gamma}(t)) = \min(s,t) \Gamma$.
\end{lemma}

\begin{remark}
This multivariate FCLT extends the univariate counterpart from, e.g., Billingsley in~\cite{Billingsley68}. 
For that, they prove in \cite{Aue09} that the univariate FCLT's are sufficient to establish the multivariate version, using Cram{\'e}r-Wold and the univariate tightness of the corresponding processes.

The cited version in Lemma~\ref{lemma:Aue_FCLT} differs in two small details from the original Theorem A.1 in \cite{Aue09}. First, it is less general as they assume $f:\R^{d' \times \infty} \rightarrow \R^d$ (and consequently $f^{\Delta}:\R^{d' \times \infty} \rightarrow \R^d$) as well as $(\epsilon_j, j \in \Z)$ to be an iid sequence of random vectors with values in $\R^{d'}$. In our case $d'=1$ is sufficient.

Further, note that we adapted \eqref{eq:repr3} from originally being $u_j^{(\Delta)} ={f}^{(\Delta)}({\epsilon}_{j},...,{\epsilon}_{j-\Delta})$. Indeed, it is straightforward to show that the proof of \cite{Aue09} still holds with this modification.  
\end{remark}

\begin{proof}{\bf of Theorem~\ref{thm:augm_GARCH_pq_asympt_qn}.}
The proof consists of four steps. We first show that the process $(X_t)$ fulfils the conditions required for having a Bahadur representation of the sample quantile, second, that a similar representation holds for the r-th absolute centred sample moment, third, that the conditions for an FCLT (Lemma~\ref{lemma:Aue_FCLT}) are fulfilled, which we then use in the fourth step to conclude the multivariate FCLT. 

{\sf Step 1: Bahadur representation of the sample quantile - conditions.}\\
The Bahadur representation of the sample quantile for a GARCH($p$, $q$) process is well known and can be obtained as a special case for Bahadur representations for processes with a certain dependence structure, see e.g. \cite{Kulik06} and references therein. Here we want to establish the Bahadur representation for sample quantiles from augmented GARCH($p$, $q$) processes of the $(Lee)$ family.
We will use the Bahadur representation for general NED processes (see Theorem~1 in \cite{Wendler11}). For the ease of comparison, we adapt some of the notation of Theorem~1 in \cite{Wendler11}. It holds under some conditions that we need to verify:
\begin{itemize}
\item[-] Choosing the bivariate function $g(x,t) := \1_{(x \leq t)}$, the non-negativity, boundedness, measurability, and non-decreasingness in the second variable, are straightforward. The function $g$ also satisfies the variation condition uniformly in some neighbourhood of $q_X(p)$ if it is Lipschitz-continuous (see Example 1.5 in \cite{Wendler11}). But the latter follows from condition $(C_2^{~'})$. 
\item[-] The differentiability of $\E[g(X,t)]= F_X(t)$ and positivity of its derivative at $t=q_X(p)$ are given by condition $(P)$ at $q_X(p)$. 
\item[-] Equation (12) in \cite{Wendler11} is fulfilled as, by our assumption $(C_2^{~'})$, $F_X$ is twice differentiable in $q_X(p)$ (see Remark 2, \cite{Wendler11}).
\item[-] The stationarity of the process follows from assumption $(P_{\max(1,r/\delta)})$ or $(L_r)$, respectively, and Lemma 1 of \cite{Lee14}.
\item[-] Lastly, let us verify that the process $(X_t)$ is $L_1$-NED with polynomial rate.
Denoting, for $s \leq t$, the sigma-algebra $\mathcal{F}_{s}^t = \sigma(\epsilon_s, ...,\epsilon_t)$, we can write for any integer $\Delta \geq 1$ 
\[ \| X_t - \E[X_t \vert  \mathcal{F}_{t-\Delta}^{t+\Delta}] \|_2^2 =  \| \sigma_t - \E[\sigma_t \vert \mathcal{F}_{t-\Delta}^{t+\Delta}] \|_2^2 ~~\sqrt{\E[\epsilon_t^2]}. \]
But $\E[ \epsilon_t^2]< \infty$ since $(M_r)$ holds. Notice that the property of being geometrically $L_2$-NED, $\| \sigma_t -\E[\sigma_t \vert \mathcal{F}_{t-\Delta}^{t+\Delta}] \|_2 = O(e^{-\kappa \Delta})$ for some $\kappa>0$, implies $L_1$-NED with polynomial rate, as
\begin{equation} \label{eq:NED-condition}
\| \sigma_t -  \E[\sigma_t \vert \mathcal{F}_{t-\Delta}^{t+\Delta}] \|_1 \leq \| \sigma_t - \E[\sigma_t \vert \mathcal{F}_{t-\Delta}^{t+\Delta}] \|_2 = O(e^{-\kappa \Delta}) = O(\Delta^{-(\beta+3)}),
\end{equation}
for some $\beta >3$. So it suffices showing that $\sigma_t$ is geometrically $L_2$-NED. For the polynomial GARCH, it follows from Corrollary 1 in \cite{Lee14}, which can be applied as $(A)$ and $(P_1)$ hold. 
For the exponential GARCH case, it follows from Corrollary 3 in \cite{Lee14} as $(A)$ and $(L_r)$ hold. 
\end{itemize}
Thus, we can use Theorem 1 of \cite{Wendler11} and write, as $n \rightarrow \infty$,
\begin{equation} \label{eq:Bahadur_qn_Wendler}
q_n(p)- q_X(p) + \frac{F(q_X(p))-F_n(q_X(p))}{f_X(q_X(p))} = o_P(1/\sqrt{n}).
\end{equation}
Note that we do not use the exact remainder bound as in \cite{Wendler11} as for our purposes $o_P(1/\sqrt{n})$ is enough.

{\sf Step 2: Representation of the r-th absolute centred sample moment -conditions.}\\
The representation being given in Proposition~\ref{prop:Abs_central_mom_asympt_garch} under some conditions, we only need to check that we fulfil them.
\begin{itemize}
\item[-] The stationarity of the process is satisfied under $(P_{max(1,r/\delta)})$ or $(L_r)$ as observed in Step 1.
\item[-] For the moment condition, short-memory property and ergodicity, we simply verify that the conditions for a CLT of $X_t^r$ (or $\lvert X_t \rvert^r$) are fulfilled,
distinguishing between the polynomial and exponential case.
Conditions $(M_r), (A), (P_{max(1,r/\delta)})$ in the polynomial case, and $(M_r), (A), (L_{r})$ in the exponential case respectively, imply the CLT, using Corollary 2 and 3 in \cite{Lee14}, respectively.
\end{itemize}

{\sf Step 3: Conditions for applying the FCLT}\\
In our case we want to apply Lemma~\ref{lemma:Aue_FCLT} in a three-dimensional version. This simplifies the computation, and by applying the continuous mapping theorem we will finally get back a two-dimensional representation. This will be made explicit in Step~4.

Therefore, let us define, anticipating its use in Step 4 for the FCLT of $U_n (X):=\frac{1}{n} \sum_{j=1}^n u_j$,
\[ u_j = \begin{pmatrix} X_j \\ \lvert X_j  \rvert^r - m(X,r) \\ \frac{p - \1_{(X_j \leq q_X(p))}}{f_X(q_X(p))}  \end{pmatrix}. \]
We need to verify that $u_j$ fulfils~\eqref{eq:repr1}: $\E[u_j] =0$ holds by construction.
$\E[ \lvert X_j\rvert^{2r}]< \infty$ is guaranteed since $\lvert X_t \rvert^{r}$ satisfies a CLT (see Step 2), thus also $\E[u_j^2] < \infty$.
Finally $X_j = {f}({\epsilon}_j,{\epsilon}_{j-1},...)$ follows from Lemma 1 in \cite{Lee14}, as we assume $(A)$.
Hence, this latter relation also holds for functionals of $X_j$, so for $u_j$, i.e. \eqref{eq:Aue09_representation_of_vector} is fulfilled.

Then, we define a $\Delta$-dependent approximation $u_0^{(\Delta)}$ satisfying \eqref{eq:repr3} and~\eqref{eq:repr2}.
Denote, for the ease of notation, $X_{0\Delta} := \E[X_0 \vert \mathcal{F}_{-\Delta}^{+\Delta}]$, and set $u_0^{(\Delta)} = \begin{pmatrix}
X_{0\Delta} \\ \E[\lvert X_0  \rvert^r \vert \mathcal{F}_{-\Delta}^{+\Delta}] - m(X,r) \\ \frac{p - \1_{(X_{0\Delta} \leq q_X(p))}}{f_X(q_X(p))} 
\end{pmatrix}$ with $\mathcal{F}_s^t = \sigma({\epsilon}_s,...,{\epsilon}_t)$ for $s\leq t$.
Thus, \eqref{eq:repr3} is fulfilled by construction. Let us verify~\eqref{eq:repr2}. We can write
\begin{align*}
& \sum_{\Delta \geq 1} \| u_0 - u_0^{(\Delta)} \|_2  
\\ &= \sum_{\Delta \geq 1} \E \left[ \left(X_0- X_{0\Delta} \right)^2 + \left( \lvert X_0\rvert^r - \E[\lvert X_0  \rvert^r \vert \mathcal{F}_{-\Delta}^{+\Delta}] \right)^2 + \frac{1}{f_X^2(q_X(p))} \left(-\1_{(X_0 \leq q_X(p))} + \1_{(X_{0\Delta} \leq q_X(p))}   \right)^2 \right]^{1/2} 
\\ & \leq   \sum_{\Delta \geq 1} \left( \| X_0- X_{0\Delta} \|_2 + \| \lvert X_0\rvert^r - \E[\lvert X_0  \rvert^r \vert \mathcal{F}_{-\Delta}^{+\Delta}] \|_2 + \frac{1}{f_X(q_X(p))} \left\| -\1_{(X_0 \leq q_X(p))} +  \1_{(X_{0\Delta} \leq q_X(p))}  \right\|_2 \right). \numberthis \label{eq:ineq1}
\end{align*}
Obviously, a sufficient condition for \eqref{eq:ineq1} is the finiteness of its summands. If it holds that each summand is geometrically $L_2$-NED, then its sum will be finite.
E.g. assuming that $X_0$ is geometrically $L_2$-NED, i.e. $ \| X_{0}  - X_{0\Delta} \|_2 =  O(e^{- \kappa \Delta}) $ for some $\kappa >0$, 
it follows that $\sum_{\Delta \geq 1} \| X_{0}  - X_{0\Delta} \|_2  < \infty$.

The condition of geometric $L_2$-NED of $X_0$ and $\lvert X_0^r \rvert$ is satisfied, on the one hand in the polynomial case under $(M_r), (A)$ and $(P_{\max(1,r/\delta)})$ via Corollary 2 in \cite{Lee14}, on the other hand in the exponential case under $(M_r), (A)$ and $(L_r)$ via Corollary 3 in \cite{Lee14}. 
Thus, as $X$ is geometric $L_2$-NED this follows also for its bounded functional $\1_{(X_0 \leq q_X(p))}$ using Lemma 3.5 in \cite{Wendler11} as we showed already in Step 1 that this functional satisfies the variation condition. 
The result in the case of an indicator function goes back to \cite{Philipp77}).

{\sf Step 4: Multivariate FCLT}\\
Having checked the conditions for the FCLT of Lemma~\ref{lemma:Aue_FCLT} in Step 3, we can apply a trivariate FCLT for $u_j$:

Using the Bahadur representation \eqref{eq:Bahadur_qn_Wendler} of the sample quantile (ignoring the rest term for the moment), we can state:
\begin{equation} \label{eq:asympt_MAD_trivariate_normal}
 \sqrt{n} \frac{1}{n} \sum_{j=1}^{[nt]} u_j =   \sqrt{n}~t \begin{pmatrix} \bar{X}_{[nt]}  \\ \frac{1}{[nt]} \sum_{j=1}^{[nt]} \lvert X_j \rvert^r - m(X,r) \\ \frac{p-F_{[nt]}(q_X(p))}{f_X(q_X(p))} \end{pmatrix} \overset{D_3[0,1]}{\rightarrow}  \textbf{W}_{\tilde{\Gamma}^{(r)}} (t) \quad \text{~as~} n \rightarrow \infty,
\end{equation}
where $\textbf{W}_{\tilde{\Gamma}^{(r)}}(t), t \in [0,1]$ is the 3-dimensional Brownian motion with covariance matrix ${\tilde{\Gamma}^{(r)}} \in \R^{3\times 3}$, i.e. the components ${\tilde{\Gamma}^{(r)}}_{ij}, 1\leq i,j \leq 3$, satisfy the same dependence structure as for the random vector $(U,V,W)^T$ described in $(D)$, with all series being absolutely convergent. 
By the multivariate Slutsky theorem, we can add $\begin{pmatrix} 0 \\ 0 \\ R_{[nt],p} \end{pmatrix}$ to the asymptotics in \eqref{eq:asympt_MAD_trivariate_normal} without changing the resulting distribution (as $\sqrt{n} R_{[nt],p} \overset{p}\rightarrow 0$). Hence, as $n \rightarrow \infty$,
\begin{equation} \label{eq:conv_semifinal}
\hspace*{-3ex} \sqrt{n}~t \begin{pmatrix} \bar{X}_{[nt]}  - \mu \\ \frac{1}{[nt]} \sum_{j=1}^{[nt]} \lvert X_j \rvert^r - m(X,r) \\ \frac{p-F_{[nt]}(q_X(p))}{f_X(q_X(p))} \end{pmatrix} +  \sqrt{n}~t \begin{pmatrix} 0 \\ 0 \\ R_{[nt],p} \end{pmatrix} =  \sqrt{n}~t \begin{pmatrix} \bar{X}_{[nt]}  - \mu \\ \frac{1}{[nt]} \sum_{j=1}^{[nt]} \lvert X_j \rvert^r - m(X,r) \\ q_{[nt]} (p) - q_X(p) \end{pmatrix}  \overset{D_3[0,1]}{\rightarrow}  \textbf{W}_{\tilde{\Gamma}^{(r)}} (t).
\end{equation} 
Recalling the representation of $\hat{m}(X,n,r)$ in Proposition~\ref{prop:Abs_central_mom_asympt_garch}, we apply to  \eqref{eq:conv_semifinal} the multivariate continuous mapping theorem using the function $f(x,y,z) \mapsto (ax+y, z)$ with $a= -r \E[(X-\mu)^{r-1} \sgn(X-\mu)^r]$. Further, by Slutsky's theorem once again, we can add to $ax+y$ a rest of $o_P(1/\sqrt{n})$ without changing the limiting distribution, to obtain, as $n \rightarrow \infty$,
\begin{align*} 
\sqrt{n}~t &\begin{pmatrix} a (\bar{X}_{[nt]}  - \mu) + \frac{1}{[nt]} \sum_{j=1}^{[nt]} \lvert X_j \rvert^r - m(X,r) + o_P(1/\sqrt{n}) \\ q_{[nt]}(p) - q_X(p) \end{pmatrix} 
\\ &=  \sqrt{n}~t \begin{pmatrix} \hat{m}(X,[nt],r) - m(X,r) \\ q_{[nt]}(p) - q_X(p) \end{pmatrix} \overset{D_2[0,1]}{\rightarrow}  \textbf{W}_{\Gamma^{(r)}} (t), \numberthis \label{eq:asymptot1_MAD}
\end{align*} 
where $\Gamma^{(r)}$ follows from the specifications of $\tilde{\Gamma}^{(r)}$ above and the continuous mapping theorem.
\end{proof}

\small
\bibliography{garch-dep}
\bibliographystyle{acm}

%
\begin{appendices} 
\section{Different Augmented GARCH models} \label{appendix}
As mentioned in the paper, we give an overview over the acronyms, authors and relation to each other of the augmented GARCH processes used.
\\ The restrictions on the parameters, if not specified differently, are $\omega \geq0, \alpha_i \geq0, -1 \leq\gamma_i \leq 1, \beta_j \geq 0$ for $i=1,...,p$, $j=1,...,q$.
\begin{itemize}
\item APGARCH: Asymmetric power GARCH, introduced by Ding et al. in \cite{Ding93}. One of the most general polynomial GARCH models.
\item AGARCH: Asymmetric GARCH, defined also by Ding et al. in \cite{Ding93}, choosing $\delta =1$ in APGARCH.
\item GJR-GARCH: This process is named after its three authors Glosten, Jaganathan and Runkle and was defined by them in \cite{Glosten93}. For the parameters $\alpha_i^{\ast} , \gamma_i^{\ast}$ it holds that $\alpha_i^{\ast} = \alpha_i (1-\gamma_i)^2$ and $\gamma_i^{\ast}= 4 \alpha_i \gamma_i$.
\item GARCH: Choosing all $\gamma_i=0$ in the AGARCH model (or $\gamma_i^{\ast} = 0$ in the GJR-GARCH), gives back the well-known GARCH($p$, $q$) process by Bollerslev in \cite{Bollerslev86}.
\item ARCH: Introduced by Engle in \cite{Engle82}. We recover it by setting all $\gamma_i =\beta_j=0, \forall i,j$.
\item TGARCH: Choosing $\delta=1/2$ in the APGARCH model leads us the so called threshold GARCH (TGARCH) by Zakoian in \cite{Zakoian94}. For the parameters $\alpha_i^{+} , \alpha_i^{-}$ it holds that $\alpha_i^{+} = \alpha_i (1-\gamma_i), \alpha_i^{-}= \alpha_i (1+ \gamma_i)$.
\item TSGARCH: Choosing $\gamma_i=0$ in the TGARCH model we get, as a subcase, the TSGARCH model, named after its authors, i.e. Taylor, \cite{Taylor86}, and Schwert, \cite{Schwert89}. 
\item PGARCH: Another subfamily of the APGARCH processes is the Power-GARCH (PGARCH), also called sometimes NGARCH (i.e. non-linear GARCH) due to Higgins and Bera in \cite{Higgins92}.
\item VGARCH: The volatility GARCH (VGARCH) model by Engle and Ng in \cite{Engle93} is also a polynomial GARCH model but is not part of the APGARCH family.
\item NGARCH: This non-linear asymmetric model is due to Engle and Ng in \cite{Engle93},  
and sometimes also called NAGARCH. 
\item MGARCH: This model is called multiplicative or logarithmic GARCH and goes back to independent suggestions, in slightly different formulations, of Geweke in
\cite{Geweke86},  Pantula in \cite{Pantula86} and Milh{\o}j in \cite{Milhoj87}.
\item EGARCH: This model is called exponential GARCH, introduced by Nelson in \cite{Nelson91}.
\end{itemize}

Then we give a schematic overview of the nesting of the different models in Figure~\ref{fig:GARCH_nested}.

Lastly, we present in Table~\ref{tbl:cond_bivariate-conv-pq} how the conditions $(P_{\max{(1, r/\delta)}})$ or $(L_r)$ respectively translate for those augmented GARCH($p$,$q$) processes - this is the generalization of Table~\ref{tbl:cond_bivariate-conv-11}.
As, in contrast to Table~\ref{tbl:cond_bivariate-conv-11}, we do not gain any insight by considering the choices of $r=1$ or $r=2$, we only present the general case, $r \in \N$.

When $p \neq q$ we need to consider coefficients $\alpha_j ,\beta_j, \gamma_j$ for $j=1,...,\max{(p,q)}$. In case they are not defined, we set them equal to 0. 

\newpage

\begin{figure}[h]
%
\centering
\includegraphics[scale=0.5]{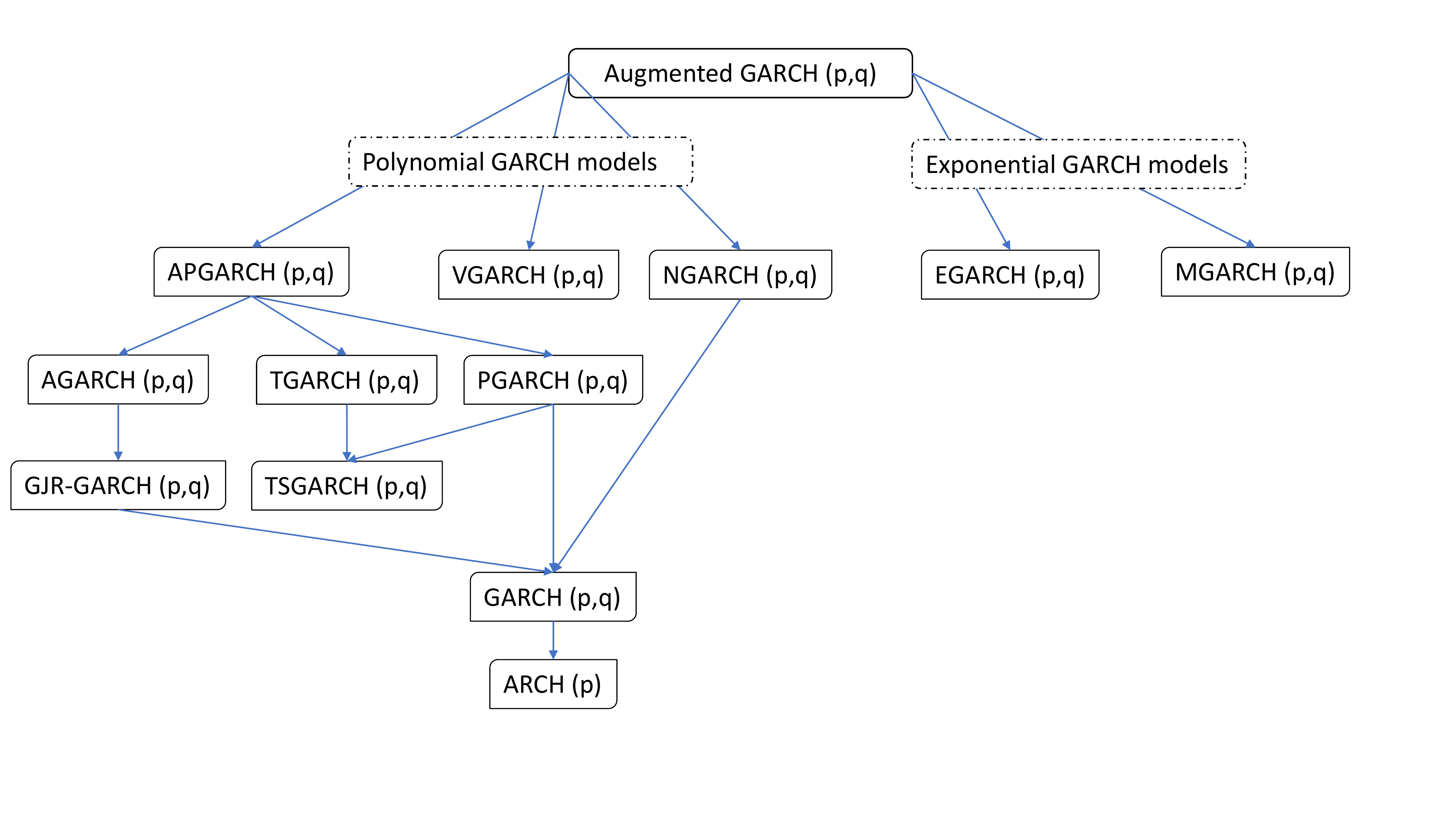}
\vspace{-1.7cm} 
\caption{\label{fig:GARCH_nested}\sf\small Schematic overview of the nesting of some augmented GARCH($p$, $q$) models.}
\end{figure}

\begin{table}[H]
{
\caption{\label{tbl:cond_bivariate-conv-pq} \sf\small  Conditions $(P_{\max{(1, r/\delta)}})$ or $(L_r)$ respectively translated for different augmented GARCH($p$,$q$) models for the general r-th absolute centred sample moment, $r \in \N$.}
\vspace{-3ex}
\begin{center}
\begin{tabular}{l | l }
\hline
\\[-1.5ex]
augmented \\ GARCH ($p$, $q$) & $r \in \N$  \\
\hline\hline
\\[-1.5ex]
\\[-2ex] APGARCH &  $\sum_{j=1}^{max(p,q)} \E[\lvert \alpha_j \left(\lvert \epsilon_0\rvert - \gamma_j \epsilon_{t-j}\right)^{2 \delta} + \beta_j \rvert^r]^{1/r} <1$ 
\\[1.5ex] AGARCH & $\sum_{j=1}^{max(p,q)} \E[\lvert \left( \alpha_j \lvert \epsilon_0\rvert - \gamma_j \epsilon_{t-j}\right)^{2 } + \beta_j \rvert^r]^{1/r} <1$ 
\\[1.5ex] GJR-GARCH & $\sum_{j=1}^{max(p,q)} \E[\lvert \alpha_j^{\ast} \epsilon_0^2 + \beta_j + \gamma_j^{\ast} \max(0,-\epsilon_0^2) \rvert^r]^{1/r} <1$
\\[1.5ex] GARCH & $\sum_{j=1}^{\max(p,q)} \E[(\alpha_j \epsilon_0^2 + \beta_j)^r]^{1/r} < 1 $ 
\\[1.5ex] ARCH & $\sum_{j=1}^{\max(p,q)} \alpha_j \E[ \epsilon_0^{2r}]^{1/r} < 1 $ 
\\[1.5ex] TGARCH & $\sum_{j=1}^{\max(p,q)} \E[\lvert \alpha_j \lvert \epsilon_{t-j} \rvert - \alpha_j \gamma_j \epsilon_{t-j} + \beta_j \rvert^r]^{1/r} < 1$
\\[1.5ex] TSGARCH & $\sum_{j=1}^{\max(p,q)} \E[\lvert \alpha_j \lvert \epsilon_{t-j} \rvert + \beta_j \rvert^r]^{1/r} < 1$
\\[1.5ex] PGARCH & $\sum_{j=1}^{max(p,q)} \E[\lvert \alpha_j \lvert \epsilon_0 \rvert +\beta_j\rvert^{2r}]^{1/(2r)} < 1$ 
\\[1.5ex] VGARCH & $\sum_{j=1}^q \beta_j < 1$
\\[1.5ex] NGARCH & $\sum_{j=1}^{\max(p,q)} \E[\lvert \alpha_j (\epsilon_0 + \gamma_j)^2 + \beta_j \rvert^r]^{1/r} < 1$ 
\\[1.5ex] MGARCH & $\E[\exp(4r \sum_{i=1}^p  \lvert \omega/p + \alpha_i \log(\epsilon_0^2)  \rvert^2)] < \infty$ and $\sum_{j=1}^q \lvert \beta_j \rvert <1$
\\[1.5ex] EGARCH & $\E[\exp(4r \sum_{i=1}^p  \lvert \omega/p + \alpha_i (\lvert\epsilon_0\rvert - \E\lvert \epsilon_0 \rvert) + \gamma_i \epsilon_0 \rvert^2] < \infty$ and $\sum_{j=1}^q \lvert \beta_j \rvert <1$
\\  [1ex]
\hline
\end{tabular}
\end{center}
}
\end{table}

\end{appendices}

\end{document}